\documentclass[12pt]{article}
 \usepackage[utf8]{inputenc}

      \usepackage{latexsym}
         \usepackage[reqno, namelimits, sumlimits]{amsmath}
         \usepackage{amssymb, amsfonts}
         \usepackage{amsthm}
           
 \newtheorem{theorem}{Theorem}[section]

 \newtheorem{remarks}[theorem]{Remarks}
 
 \newtheorem{pro}[theorem]{Proposition}

\author{ Gregory Seregin
}

\title{Remarks on Type II blowups of solutions to the Navier-Stokes equations}

\author{G.~Seregin\footnote{University of Oxford, Mathematical Institute, OxPDE, Oxford, UK and St Petersburg Department of Steklov Mathematical Institute, RAS, Russia, email address: \texttt{seregin@maths.ox.ac.uk}}
}

\begin{document}


\maketitle

\rightline{\it Dedicated to Vladimir Sverak}

\begin{abstract}  In the note, the Euler scaling is used to study a certain scenario of potential Type II blowups of solutions to the Navier-Stokes equations.
\end{abstract}


{\bf Data availability statement}
Data sharing not applicable to this article as no datasets were generated or analysed during the current study.



\setcounter{equation}{0}
\section{Type II Blowups}

Let us consider a suitable weak solution $v$ and $q$ to the classical Navier-Stokes equations, describing the flow of a viscous incompressible fluid in a unique parabolic cylinder $Q$. By definition, see \cite{CKN}, \cite{Lin}, and \cite{LS1999}, this pair of functions $v$ and $q$ has the following properties:
\begin{itemize}
	\item $v\in L_\infty(-1,0;L_2(B)),\quad \nabla v\in L_2(Q),\quad q\in L_\frac 32(Q);$
	\item 
$\partial_tv+v\cdot\nabla v-\Delta v+\nabla q=0,\quad{\rm div}\,v=0$\\
in $Q$ in the sense of distributions;
\item
for a.a. $t\in ]-1,0[$, the local energy inequality 
$$\int\limits_B\varphi(x,t)|v(x,t)|^2dx+2\int\limits^t_{-1}\int\limits_B\varphi|\nabla v|^2dxd\tau\leq 
$$
$$\leq \int\limits^t_{-1}\int\limits_B (|v|^2(\partial_t\varphi+\Delta\varphi)+v\cdot\nabla\varphi(|v|^2+2q))dxd\tau$$ holds
for all smooth non-negative functions $\varphi$ vanishing in a vicinity of a parabolic boundary of the cylinder $Q$. 
\end{itemize}
Here and in what follows, the standard notation is used:
$Q(r)=B(r)\times ]-r^2,0[$ is a parabolic cylinder 
and $B(r)$ is a spatial ball of radius $r$ centred at the origin $x=0$, $B=B(1)$, and $Q=Q(1)$.

Our aim is to find assumptions for the space-time origin $z=(x,t)=0$ to be a regular point of $v$, i.e., for the existence of a 
positive number $r$ such that $v\in L_\infty(Q(r))$. We say that $z=0$ is a blowup point of $v$ or a singular point of $v$ if it is not a regular point of $v$. We shall distinguish between Type I and Type II blowups, see \cite{Seregin2010} for the original definition.
 It is said that $z=0$ is a Type I blowup of $v$ if 
\begin{equation}
	\label{DefTypeI}
	g=\max\{\limsup\limits_{r\to0}A(v,r),\limsup\limits_{r\to0}E(v,r),\limsup\limits_{r\to0}C(v,r)\}<\infty.
\end{equation}
Here, the following energy quantities are used:
$$A(v,r)=\frac1r\sup\limits_{-r^2<t<0}\int\limits_{B(r)}|v(x,t)|^2dx,\quad
E(v,r)=\frac 1r\int\limits_{Q(r)}|\nabla v|^2dz,
$$
$$C(v,r)=\frac 1{r^2}\int\limits_{Q(r)}|\nabla v|^2dz.
$$
The important property of the above  quantities is that all of them are invariant with respect to the Navier-Stokes scaling, i.e., if $v^\lambda(y,s)=\lambda v(x,t)$, and $q^\lambda(y,s)=\lambda^2 q(x,t)$ with $x=\lambda y$ and $t=\lambda^2s$, then, for example, $A(v^\lambda,r)=A(v,\lambda r)$, etc.

According to the celebrated Cafarelli-Kohn-Nirenberg theory, which is, in fact, $\varepsilon$-regularity theory developed for the specific case of the Navier-Stokes equations, $z=0$ is a regular point of $v$ if $g$ is sufficiently small, see details in \cite{Seregin2015}. The question about Type I blowups is whether the boundedness of $g$ allows blowups or not. It is still open.

There is a sufficient condition for the point $z=0$ to be a Type I blowup of $v$, see more details in papers  \cite{SZ2006} and   \cite{SS2018}, and it is as follows:
\begin{equation}
	\label{TypeI}
\limsup\limits_{R\to0}M_\kappa^{s,l}(v,R)=\limsup\limits_{R\to0}\frac 1{R^\kappa}\int\limits_{-R^2}^0\Big(\int\limits_{B(R)}
|v|^sdx\Big)^\frac lsdt<\infty\end{equation}
with
$$\quad \kappa=l\Big(\frac 3s+\frac 2l-1\Big),$$
where numbers $s\geq1$ and $l\geq1$ obey the restrictions
\begin{equation}\label{add-cond-sl}
1>\frac 3s+\frac 2l-1>\frac 12+\max\Big\{\frac 1{2l},\frac 12-\frac 1l\Big\}.\end{equation}

In fact, the above sufficient condition is true for a wider set of numbers $s$ and $l$. The following statement is a simple extension  of the corresponding result in \cite{SS2018}.

\begin{theorem}
\label{extension} Let $v$ and $q$ be a suitable weak solution to the Navier-Stokes equations in $Q$. Assume that real number $s\geq 1$
and $l\geq 1$	satisfy the restriction
\begin{equation}
	\label{weak13gen}
	1>\frac 3s+\frac 2l-1 (\Longleftrightarrow l>\kappa).
\end{equation} Suppose that the condition 
\begin{equation}
\label{TypeIA}
\mathcal M=\sup\limits_{0<R<1}M^{s,l}(R)<\infty,	
\end{equation}
see \eqref{TypeI} for the definition of $M^{s,l}(R):=M^{s,l}_\kappa(v,R)$, holds. Then, there exist positive constants $\varepsilon_0$, $c$, and $\alpha_0<1$ such that
\begin{equation}
	\label{decayA}
\mathcal E(r)\leq cr^{\alpha_0}\mathcal E(1)+c\varepsilon+c_0(s,l,\varepsilon,\mathcal M)
\end{equation}
for all $0<r\leq1$ and for all $0<\varepsilon\leq\varepsilon_0$. Moreover, $c_0$ is a continuous function of its arguments and  such that $c_0(s,l,\varepsilon,\mathcal M)\to 0$ as $\mathcal M\to0$. 

Here,
$\mathcal E(r)=E(v,r)+A(v,r)+D(q,r)$.
\end{theorem}
A proof of Theorem \ref{extension} is outlined in Appendix III. 

In what follows, we shall consider the following restrictions on numbers $s$ and $l$:
\begin{equation}
	\label{weak13}
	l>\kappa>0.
\end{equation}
The point is that if $\kappa\leq0$ and $M^{s,l}_\kappa(v,R)<\infty$ for some $R>0$   then by the Ladyzhenskaya-Prodi-Serrin condition the origin $z=0$ is a regular point of $v$.

In this note, we would like to discuss whether the origin could be a Type II blowup, i.e., a blowup which is not of Type I. So, assume that $z=0$ is a Type II blowup of $v$. According to the above sufficient condition, there must be:
$$\limsup\limits_{r\to0}M_\kappa^{s,l}(v,r)=\infty$$
 for all numbers $s$ and $l$ satisfying restriction (\ref{weak13}). Our main assumption is about how fast quantity $M^{s,l}_\kappa(v,r)$ tends to infinity as $r\to0$. To be more precise, we are going to suppose that there are a sequence $r_k\to0$ and positive numbers $\varepsilon_0$ and $m_0<1$ such that
 $$M^{s,l}_\kappa(v,r_k)\geq \frac{\varepsilon_0}{r_k^{(1-m_0)\kappa}}$$ for some numbers $s$ and $l$ satisfying condition \eqref{weak13}.
 The latter can be written in a more compact form:
 \begin{equation}
 	\label{growth}
 	M_{\kappa,m_0}^{l,s}(v,r_k):=r_k^{(1-m_0)\kappa}M^{s,l}_\kappa(v,r_k)\geq \varepsilon_0. \end{equation}

We are going to consider the following scenario of a potential Type II blowup described by (\ref{growth}), assuming in addition that 
\begin{equation}
	\label{ScenarioI}
	A_{m_1}(v,R)+D_m(q,R)+E_m(v,R)\leq c
\end{equation}
for all $0<R\leq 1$, where
$$E_m(v,r)=\frac 1{r^{2m}}\int\limits_{Q(r)}|\nabla v|^\frac 32dz, \quad A_{m_1}(v,r)=\sup\limits_{-r^2<t<0}\frac 1{r^{m_1}}\int\limits_{B(r)}|v(x,t)|^2dx,$$
$$D_m(q,r)=\frac 1{r^{2m}}\int\limits_{Q(r)}|q|^\frac 32dz,$$
and 
\begin{equation}\label{mandm0} m_1=2m-1,\qquad m=2-\alpha,\quad
 \alpha=\frac{l\Big(\frac 3s+1\Big)-(m_0-1)\kappa}{l\Big(\frac 3s+1\Big)+(m_0-1)\kappa}.
\end{equation}
It is easy to check that $\alpha >1$ and thus $m_1<m<1$.

In order to state the main result of the paper, let us introduce numbers
\begin{equation}
	\label{numbers}
\frac{1}{	p(\lambda)}=\frac \lambda 6+\frac {3(1-\lambda)} {10}, \quad \frac 1{q(\lambda)}=\frac \lambda 2+\frac {3(1-\lambda)} {10},
\end{equation}
where $0\leq \lambda\leq 1$ is a parameters. It is easy to check 
that 
$$1>\frac 3{p(\lambda)}+\frac 2{q(\lambda)}-1=\frac 12>0$$
for all $0\leq \lambda\leq 1$.
One of our main results can be stated as follows.
\begin{pro} 
\label{ancientsolution}	Suppose that the pair $v$ and $q$ is a suitable weak solution to the Navier-Stokes equations in the unit space-time cylinder $Q$.

Assume that, for some numbers $s$ and $l$ satisfying restrictions (\ref{weak13})  and the additional inequalities
\begin{equation}
	\label{numberscond}s <p(\lambda),\quad l<q(\lambda)\end{equation}
with some $0\leq \lambda\leq 1$, the growth condition \eqref{growth} holds.

Assume further that $v$ and $q$ satisfy bound \eqref{ScenarioI} with numbers $m$ and $m_1$
defined in \eqref{mandm0}.

Then, there are two functions $u$ and $p$ defined in $Q_-=\mathbb R^3\times]-\infty,0[$, with the following properties:
	$$\sup\limits_{a>0}\Big[\sup\limits_{-a^2<s<0}\frac 1{a^{m_1}}\int\limits_{B(a)}|u(y,s)|^2dy+\frac 1{a^{2m}}\int\limits_{Q(a)}|p|^\frac 32dyds+$$
 \begin{equation}
	\label{basicestimates}
+\frac 1{a^{m}}\int\limits_{Q(a)}|\nabla u|^2dyds\Big]\leq c<\infty;\end{equation}

\begin{equation}
	\label{Euler}\partial_tu+u\cdot\nabla u+\nabla p=0, \quad{\rm div}\,u=0
\end{equation}in $Q_-=\mathbb R^3\times ]-\infty,0[$;
 
 for a.a. $\tau_0\in ]-\infty,0[$, the local energy inequality
$$\int\limits_{\mathbb R^3}|u(y,\tau_0)|^2\varphi(y,\tau_0)dy\leq$$
\begin{equation}
	\label{enerylocal}
\leq \int\limits_{-\infty}^{\tau_0}\int\limits_{\mathbb R^3}(|u|^2(\partial_s\varphi+\Delta_y\varphi)+u\cdot\nabla_y\varphi(|u|^2+2
p))dyd
\tau
\end{equation}
holds for non-negative $\varphi\in C^\infty_0(\mathbb R^3\times \mathbb R)$;
 
 the function $u$ is not trivial in the following sense
 \begin{equation}
	\label{nontrivial}
M^{s,l}_{\kappa,m_0}(u,1)\geq \frac {\varepsilon_0}2.\end{equation} 
\end{pro}

\begin{remarks} \label{remark1}
Conditions 	\eqref{numberscond} provides a certain compactness which is needed to get property \eqref{nontrivial}. They do not seem to be optimal.
\end{remarks}

A proof of the statement is based on zooming  arguments for a certain Euler scaling and is given in the next section.

In Section 3, simple cases are discussed.
The section contains  some trivial Liouville type theorems, stating that the function $u$ of Proposition \ref{ancientsolution} cannot be non-trivial. In particular, it follows from them that in order to have a function $u$, satisfying all the statements of Proposition    \ref{ancientsolution},  $m$ must be bigger or equal to 1/2 or $u$ may not be irrotational.

In section 4, we are answering the question: what kind of Type II blowups are potentially possible under  scenario (\ref{ScenarioI}). To this end, we are going to use local energy estimates in terms of scaled energy quantitates, see Appendix II, and what kind of Type II blowup can be expected if boundedness of the quantity $M^{s,l}_{\kappa,m_0}$ takes place, see Appendix I.

The last section addresses a possible way to look for a non-trivial function $u$ in the self-similar form or in the discrete self-similar form. A certain Liouville type theorem for the self-similar profile has been proven in Proposition \ref{simpleLiouville}.

At the end of  Section 1, we would like to raise the following interesting question. Assume that we are able to construct a non-trivial solution $u$ that satisfies all the properties listed in Proposition \ref{ancientsolution}.
Can we construct a suitable weak solution $v$ and $q$ in $Q$, undergoing a blowup at $z=0$, with the help of function $u$. In the case of Type I blowups, there is an answer to this question and it is more or less positive, see papers \cite{AlBar2018} and \cite{Seregin2020}.

\setcounter{equation}{0}
\section{Proof of Proposition \ref{ancientsolution}}

In order to prove our statement, we are going to use the Euler scaling	$$v^{\lambda,\alpha}(y,\tau)=\lambda^\alpha v(x,t), \qquad q^{\lambda,\alpha}(y,
\tau)=\lambda^{2\alpha}q(x,t),$$
$$x=\lambda y,\qquad t=\lambda^{\alpha+1}\tau,$$
where $\alpha$ is defined in (\ref{mandm0}). It is  known that the Euler equations are invariant with respect to such a scaling.
As to the choice of $\lambda=
\lambda_k$, it is as follows: 
\begin{equation}
	\label{choicelambda}
\lambda_k^{\frac 12(\alpha+1)}=r_k.\end{equation}

Now, let us introduce additional notation for parabolic cylinders:
$$Q^{\lambda,\mu}(R)=\{(y,\tau):\,|y|<\lambda R,\quad -\mu R^2<\tau<0\},$$
$$
Q^{\lambda,\mu}=
Q^{\lambda,\mu}(1).$$
Then, by scaling, by the choice of $\lambda<1$, and by the choice of $\alpha>1$, we have
$$ M^{s,l}_{\kappa,m_0}(v^{\lambda,\alpha},1)=\lambda^{\alpha l-\frac {3l}s-\alpha-1}\int\limits^0_{-(\lambda^{\frac {\alpha+1}2})^2}\Big(\int\limits_{B(\lambda)}|v|^sdx\Big)^\frac lsdt\geq 
$$
$$\geq \lambda^{\alpha l-\frac {3l}s-\alpha -1}(\lambda^{\frac {\alpha+1}2})^{\kappa m_0}M^{s,l}_{\kappa,m_0}(v,\lambda^{\frac {\alpha+1}2})=M^{s,l}_{\kappa,m_0}(v,r_k).
$$
So, finally,
\begin{equation}
	\label{non-trivialk}
M^{s,l}_{\kappa,m_0}(v^{\lambda,\alpha},1)\geq \frac {\varepsilon_0}2\end{equation}
 for all $k$.

Now, after making the Euler scaling in the Navier-Stokes equitations, multiplying them by $v^{\lambda,\alpha}$, and integrating by parts, we arrive  at the following identity:
$$\int\limits_{B(a)}\partial_\tau v^{\lambda,\alpha}(y,\tau)\cdot w(y)dy=\int\limits_{B(a)}v^{\lambda,\alpha}(y,\tau)\otimes v^{\lambda,\alpha}(y,\tau):\nabla w(y)dy+$$
\begin{equation}
	\label{deriveintime}
	+\lambda^{\alpha-1}\int\limits_{B(a)}v^{\lambda,\alpha}(y,\tau)\cdot \Delta w(y)dy+\int\limits_{B(a)}q^{\lambda,\alpha}(y,\tau){\rm div}\,w(y)dy,	\end{equation}
which is valid for any $w\in C^\infty_0(B(a))$. So,  applying of various multiplicative inequalities,
we find
$$\Big|\int\limits_{B(a)}\partial_\tau v^{\lambda,\alpha}(y,\tau)\cdot w(y)dy\Big|\leq $$$$\leq c(a)\Big(\int\limits_{B(a)}(|v^{\lambda,\alpha}(y,\tau)|^3+|q^{\lambda,\alpha}(y,\tau)|^\frac 32)dy\Big)^\frac 23\Big(\int\limits_{B(a)}|\nabla^2w(y)|^\frac 32dy\Big)^\frac 23.$$
Furthermore, H\"older inequality leads us to the bound
$$\|\partial_\tau v^{\lambda,\alpha}(\cdot,\tau)\|_{H^{-2}(B(a))}\leq c(a)(\|v^{\lambda,\alpha}(\cdot,\tau)\|_{3,B(a)}^2+\|q^{\lambda,\alpha}(\cdot,\tau)\|_{\frac 32,B(a)})$$ and thus, as a result of integration in $\tau$, we find an important estimate for the derivative in time:
$$\|\partial_\tau v^{\lambda,\alpha}(\cdot,s)\|_{L_\frac 32(-a^2,0;H^{-2}(B(a)))}\leq c(a)(\|v^{\lambda,\alpha}\|^2_{3,Q(a)}+\|q^{\lambda,\alpha}\|_{\frac 32,Q(a)}).$$

Next, our scaling and the choice of $m$ give us:
$$A_{m_1}(v^{\lambda,\alpha},a)=
\frac {\lambda^{-3+2\alpha+m_1}}{(\lambda a)^{m_1}}\sup\limits_{-\lambda^{\alpha-1}(\lambda a)^2<t<0}\int\limits_{B(\lambda a)}|v(x,t)|^2
dx\leq A_{m_1}(v,\lambda a),$$
$$
D_m(q^{\lambda,\alpha},a)=
\frac {\lambda^{-4+2\alpha+2m}}{a^{2m}}\int\limits_{Q^{\lambda,{\lambda^\alpha-1}}(\lambda a)}|q|^\frac 32dz\leq D_m(q,\lambda a),$$
$$E_m(v^{\lambda,\alpha},a)=
\frac {\lambda^{-2+\alpha+m}}{a^{m}}\int\limits_{Q^{\lambda,{\lambda^\alpha-1}}(\lambda a)}|\nabla v|^2dz\leq E_m(v,\lambda a).$$
Hence, from (\ref{ScenarioI}), the Morrey type estimate comes out:
\begin{equation}
	\label{uniformbound}
A_{m_1}(v^{\lambda,\alpha},a)+D_m(v^{\lambda,\alpha},a)+E_m(v^{\lambda,\alpha},a)\leq c\end{equation}
for all $1/\lambda \geq a>0$.
Moreover,  the standard multiplicative inequality 
$$C_{\tilde m}(v,r):=\frac 1{r^{2\tilde m}}\int\limits_{Q(r)}|v|^3dz\leq $$
\begin{equation}
	\label{multple1}
\leq r^{\frac 32m_1+\frac 12-2\tilde m} A^\frac 34_{m_1}(v,r)(r^{m-m_1}E_m(v,r)+A_{m_1}(v,r))^\frac 34
\end{equation}  is valid for any $r>0$
and for any $\tilde m\geq m$. We let
 $v=v^{\lambda,\alpha}$, $\tilde m=m$, and $r=a$ in it and deduce from (\ref{uniformbound})  that
$$\|v^{\lambda,\alpha}\|_{3,Q(a)}\leq c(a)
$$ for all $1/\lambda\geq a>0$.
Therefore, 
$$\|\partial_\tau v^{\lambda,\alpha}(\cdot,s)\|_{L_\frac 32(-a^2,0;H^{-2}(B(a)))}+\|\nabla v^{\lambda,\alpha}\|_{2,Q(a)}+\| v^{\lambda,\alpha}\|_{2,Q(a)}\leq 
c(a).$$
By the known compactness result, the sequence $v^{\lambda_k,\alpha}$ is compact in $L_2(Q(a))$.
Using the Cantor's diagonal process, we can select a subsequence, converging
on any compact set, i.e.,
$v^{\lambda_k,\alpha}\to u$ in $L_2(Q(a))$ for any $a>0$. The latter immediately implies that
$u$ and $p
$ satisfy the Euler equations
$$\partial_tu+u\cdot\nabla u+\nabla p=0, \quad{\rm div}\,u=0$$
in $Q_-=\mathbb R^3\times ]-\infty,0[$.

By the same compactness property, one can state that $v^{\lambda_k,\alpha}\to u$ 
in $L_{s_0,2}(Q)$ for any $2\leq s_0<6$. Applying again  the known multiplicative inequality, we also have
$$\int\limits_{Q(a)}|v^{\lambda_k,\alpha}|^\frac {10}3dyd\tau\leq 
c\Big(\sup\limits_{-a^2<\tau<0}\int\limits_{B(a)}|v^{\lambda_k,\alpha}(y,\tau)|^2ds\Big)^\frac23\Big(\int\limits_{Q(a)}|\nabla v^{\lambda_k,\alpha}|^2dyd\tau+$$$$+\frac 1{a^2}\int\limits_{Q(a)}| v^{\lambda_k,\alpha}|^2dyd\tau\Big)\leq c(a)$$
for all $a>0$. The latter yields the  convergence in $L_{s_1,loc}$ with $3\leq s_1< \frac {10}3$.

Now, we let, for small positive $\varepsilon$,  
$$s_0=6-\varepsilon,\quad s_1=\frac {10}3-\varepsilon,$$$$ \frac 1{\tilde p(\lambda)}=\frac \lambda {s_0}+\frac {1-\lambda}{s_1}>\frac 1{p(\lambda)},\quad \frac1{\tilde q(\lambda)}=\frac \lambda 2+\frac {1-\lambda}{s_1}>\frac 1{q(\lambda)}. $$
By H\"older inequality
$$\|v^{\lambda_k,\alpha}-u\|_{\tilde p(\lambda),\tilde q(\lambda),Q}\leq \|v^{\lambda_k,\alpha}-u\|_{s_0,2,Q}^\lambda\|v^{\lambda_k,\alpha}-u\|_{s_1,Q}^ {1-\lambda}\to0$$
as $k\to\infty$.

It remains to pick up $\varepsilon$ small enough  to provide $s\leq \tilde p(\lambda)<p(\lambda)$ and $l\leq \tilde q(\lambda)<q(\lambda)$ and to see that $v^{\lambda_k,\alpha}\to $ in $L_{s,l}(Q)$. The latter implies \eqref{nontrivial}.


Next, by the scaling in the local energy inequality, we have, for $-a^2<\tau_0<0$,
$$\int\limits_{B(a)}|v^{\lambda_k,\alpha}(y,\tau_0)|^2\varphi(y,\tau_0)dy+2\lambda_k^{\alpha-1}\int\limits_{-a^2}^{\tau_0}\int\limits_{B(a)}\varphi|\nabla_yv^{\lambda_k,\alpha}|^2dyd\tau\leq$$
$$\leq \int\limits_{-a^2}^{\tau_0}\int\limits_{B(a)}(|v^{\lambda_k,\alpha}|^2(\partial_s\varphi+\Delta_y\varphi)+v^{\lambda_k,\alpha}\cdot\nabla_y\varphi(|v^{\lambda_k,\alpha}|^2+2
q^{\lambda_k,\alpha}))dyd\tau$$
for any smooth non-negative functions $\varphi$ having compact support in $B(a)\times]-a^2,a^2[$. 
So, in the limit as $k\to \infty$, we can get (\ref{basicestimates}) and (\ref{enerylocal}).


\setcounter{equation}{0}
\section{Simple Cases}

By basic estimates (\ref{basicestimates}), we have for $a>1$
$$\sup\limits_{-(\frac a2)^2<\tau<0}\int\limits_{B(a/2)}|u(y,\tau)|^2dy\leq ca^{m_1}$$
So, if $m_1<0$ which means that $m<\frac 12$,
then, sending $a\to \infty$, we conclude that $u=0$. Therefore, Type II singularities (\ref{growth}) under restrictions (\ref{ScenarioI}) cannot occur. In other words, if such a singularity exists, then 
$$\sup\limits_{0<R\leq 1}A_{m_1}(R)+D_m(R)+E_m(R)=\infty.$$
So, in what follow we shall assume that 
$m\geq\frac 12$.  The latter gives
the following lower bound on $m_0$:
\begin{equation}
	\label{lowerbndm0}
	m_0\geq \frac{2\Big(\frac{6}s+\frac {5}l-3\Big)}{5\Big(\frac 3s+\frac 2l-1\Big)}.
\end{equation}
It is check that the right hand side of (\ref{lowerbndm0}) satisfies
$$\frac{2\Big(\frac{6}s+\frac {5}l-3\Big)}{5\Big(\frac 3s+\frac 2l-1\Big)}\leq \frac 45.$$
So, (\ref{lowerbndm0}) holds if 
$$m_0\geq \frac 45.$$

Another question to ask is about whether $u$ can be an irrotational flow.
Assume that it is so, i.e., ${\rm rot}\,u=0$. Therefore, there is a potential $\varphi$ such that $u=\nabla \varphi$. Then, by incompressibility condition, $u$ and $\varphi$ are harmonic functions.
Let us fix $t_0<0$. Our aim is to show that the function $x\mapsto u(x,t_0)$ is bounded. It is locally bounded as $u$ is harmonic. In order to prove that $u(x,t_0)$ is bounded as $x\to\infty$, let us assume that $-|x_0|^2<t_0$. Then, by harmonicity of $u$, we have
$$|u(x_0,t_0)|^2\leq \frac 1{|B(x_0,|x_0|)|}\int\limits_{B(x_0,|x_0|)}|u(x,t_0)|^2dx\leq$$$$\leq  2^3\frac 1{|B(2|x_0|)|}\int\limits_{B(2|x_0|)}|u(x,t_0)|^2dx\leq$$$$\leq c|x_0|^{m_1-3}\sup\limits_{-(2|x_0|)^2<t<0}\frac 1{(2|x_0|)^{m_1}}\int\limits_{B(2|x_0|)}|u(x,t)|^2dx.$$
By (\ref{basicestimates}),
$$|u(x_0,t_0)|\leq c|x_0|^{m_1-3}$$ as long as $-|x_0|^2<t_0$. It implies that $u(x,t_0)=c(t_0)$ and thus $u(x,t_0)=0$ for all $x\in \mathbb R^3$. We conclude that $u$ cannot be irrotational.

\setcounter{equation}{0}
\section{Does scenario (\ref{ScenarioI}) admit Type II blowups?}

As one can observe in Appendix I, uniform boundedness of functionals of type $M^{s,l}_\kappa(v,r)$ with a certain index $\kappa$ does not exclude Type II blowups of special form and in turn, see Appendix II, provides uniform boundedness of certain scaled energy quantities. We are going to use it.

Let $0<m<1$ be fixed. Assume that numbers $s\geq 1$ and $l\geq 1$ satisfy assumptions \eqref{add-cond-sl}.

Suppose further that
\begin{equation}
	\label{TypeI'}
	\sup\limits_{0<R<1}M^{s,l}_{\kappa_m}(v,R)<\infty,
\end{equation}
where 
$$\kappa_m=m\kappa+q(1-m),$$

$$
\kappa=l\Big(\frac 3s+\frac 2l-1\Big),
\quad q=2l\Big(\frac 3{s}+\frac 2l-\frac 32\Big).$$

It is interesting to observe that the number $m_0$ satisfying the identity
$\kappa m_0=\kappa_m$ belongs to the interval $]0,1[$ and thus
$$M^{s,l}_{\kappa,m_0}(v,R)=M^{s,l}_{\kappa_m}(v,R).$$
Then, as it follows from Appendix II, we have the following estimate
$$A_{m_1}(v,r)+E_m(v,r)+ D_m(q,r)\leq A_m(v,r)+E_m(v,r)+\widetilde D_m(q,r)\leq c
$$ for all $0<r<1$, where $\widetilde D_m(q,r)=r^{m-1}D_m(r)$. The latter estimate takes place 
 since $m_1=2m-1<m$ and $2m<m+1$.

\setcounter{equation}{0}
\section{Self-Similarity}

Here, we are looking for the function  $u$ of  Proposition \ref{ancientsolution} in the following form
$$u(x,t)=\frac 1{(-t)^{\frac \alpha{\alpha+1}}}U(y,\tau),\quad p(x,t)=\frac 1{(-t)^{\frac {2\alpha}{\alpha+1}}}P(y,\tau),
$$
where 
$$y=\frac x{(-t)^{\frac 1{\alpha+1}}}, \quad \tau =-\ln(-t)
$$
and the number $\alpha>1$ is defined in (\ref{mandm0}). Then the Euler equations take the form 
\begin{equation}
	\label{special-form}
	\partial_\tau U+\frac \alpha{\alpha+1}U+\frac 1{\alpha+ 1}y\cdot\nabla U+U\cdot\nabla U+\nabla P=0,\quad{\rm div}\,U=0
\end{equation}
for  $(y,\tau)\in\mathbb R^3\times \mathbb R$.	

If $U(y,\tau)=U(y)$, then $u$ is a self-similar solution since
\begin{equation}\label{self-similar}
u(x,t)=\lambda^\alpha u(\lambda x,\lambda^{\alpha+1}t)
\end{equation}
for all positive $\lambda$.

If $U(y,\tau)=U(y,\tau+S_0)$ for any $(y,\tau)$ and for some positive $S_0$, then $u$ is a discrete self-similar solution for which (\ref{self-similar}) holds with $\lambda=e^{\frac  {S_0}{\alpha+1}}$.

Now, our goal is to describe additional restrictions on $U$ that follows from
(\ref{basicestimates}). First, we have
$$A_{m_1}(u,a)=\frac 1{a^{m_1}}\sup\limits_{-a^2<t<0}\int\limits_{B(a)}|u(x,t)|^2dx=$$
$$=\frac 1{a^{m_1}}\sup\limits_{-\ln a^2<\tau<\infty}e^{\frac \tau{1+\alpha}(2\alpha-3)}\int\limits_{B(ae^{\frac \tau{1+\alpha}})}|U(y,\tau)|^2dy.
$$
So, we have 
\begin{equation}
	\label{1st}
	\sup\limits_{a>0}\sup\limits_{-\ln a^2<\tau<\infty}b^{-m_1}(a,\tau)\int\limits_{B(b(a,\tau))}|U(y,\tau)|^2dy\leq c<
	\infty\end{equation}
with $b(a,\tau)=a e^{\frac \tau{1+\alpha}}$.

In the same way, we find
\begin{equation}\label{2nd}
\sup\limits_{a>0}E_m(u,a)=$$$$=\sup\limits_{a>0}	\int\limits^\infty_{-\ln a^2}\Big(\frac 1{b(a,\tau)}\Big)^md\tau\int\limits_{B(b(a,\tau))}|\nabla U(y,\tau)|^2dy<\infty
\end{equation}
 and
 \begin{equation}
 	\label{3rd}
 	\sup\limits_{a>0}D_m(u,a)=\sup\limits_{a>0}	\int\limits^\infty_{-\ln a^2}\Big(\frac 1{b(a,\tau)}\Big)^{2m}d\tau\int\limits_{B(b(a,\tau))}|P(y,\tau)|^\frac 32dy<\infty. \end{equation}

After change of variables in the local energy inequality, the following inequality comes out
$$\int\limits_{\mathbb R^3}\varphi(y,\tau_0)|U(y,\tau_0)|^2e^{-\frac {\tau_0m_1}{1+\alpha}}dy\leq
$$
\begin{equation}
	\label{self-loc-ineq}
	\leq\int\limits_{-\infty}^{\tau_0}\int\limits_{\mathbb R^3}e^{-\frac {\tau m_1}{1+\alpha}}\Big[|U|^2\Big(\frac 1{1+\alpha}y\cdot\nabla \varphi+\partial_\tau\varphi\Big)+U\cdot\nabla \varphi\Big(|U|^2+2P\Big)\Big]dyd\tau.
\end{equation}
The latter inequality is valid for all $\tau_0$ and for all negative $\varphi\in C^\infty_0(\mathbb R^3\times \mathbb R)$.

Let us take any non-negative function $\psi\in C^\infty_0(\mathbb R^3\times \mathbb R)$. Let $\varphi=\psi e^{\frac{\tau m_1}{1+\alpha}}$. Then $\partial_\tau\varphi=\partial_\tau\psi e^{\frac{\tau m_1}{1+\alpha}}+\psi e^{\frac{\tau m_1}{1+\alpha}}\frac {m_1}{1+\alpha} $ and inequality (\ref{self-loc-ineq}) takes the form
$$\int\limits_{\mathbb R^3}\psi(y,\tau_0)|U(y,\tau_0)|^2dy\leq
$$
\begin{equation}
	\label{self-loc-ineq1}
	\leq\int\limits_{-\infty}^{\tau_0}\int\limits_{\mathbb R^3}\Big[|U|^2\Big(\frac 1{1+\alpha}y\cdot\nabla \psi+\partial_\tau\psi+\frac {m_1}{1+\alpha}\psi\Big)+U\cdot\nabla \psi\Big(|U|^2+2P\Big)\Big]dyd\tau
\end{equation}
for all non-negative $\psi\in C^\infty_0(\mathbb R^3\times \mathbb R)$.

First, let us consider  the case of self-similarity: $U(y,\tau)=U(y)$.
Then, for any $b\geq 1$, we can take $a=b$ and $\tau=0$. Obviously, $-\ln a^2\leq \tau$ and, hence, we can deduce from (\ref{1st}) the following:
\begin{equation}
	\label{1sts-s}
	\sup\limits_{b\geq 1}\frac 1{b^{m_1}}\int\limits_{B(b)}|U(y)|^2dy\leq c<\infty.
\end{equation}
In particular, it means that our solution cannot be a non-trivial self-similar one if $m_1<0$, i.e.,
\begin{equation}
	\label{nots-s1}
	 m <\frac 12.
\end{equation}
Next, from (\ref{2nd}) it follows that
$$c\geq \int\limits^\infty_{B(ae^{\frac {-\ln a^2}{1+\alpha}})}|\nabla U|^2dy
\int\limits^\infty_{-\ln a^2}\frac 1{(ae^{\frac \tau{1+\alpha}})^m}d\tau$$
$$= \frac {\alpha+1}m \Big(a^{-\frac {\alpha-1}{\alpha+1}}\Big)^m\int\limits^\infty_{B(a^{\frac {\alpha-1}{\alpha+1}})}|\nabla U|^2dy
$$ 
for any $a>0$. The latter implies the second estimate for the self-similarity case
\begin{equation}
	\label{2nds-s}
\sup\limits_{b>0}\frac 1{b^m}\int\limits_{B(b)}|\nabla U|^2dy<\infty.
\end{equation}
In the same way, we can get the third estimate for the pressure
\begin{equation}
	\label{3rds-s}
\sup\limits_{b>0}\frac 1{b^{2m}}\int\limits_{B(b)}|P|^\frac 32dy<\infty.	\end{equation}
Now, let us evaluate the $L_3$-norm of $U$ over balls of radius $a>1$? Indeed, we have
$$\int\limits_{B(a)}|U|^3dx\leq c\Big(\int\limits_{B(a)}|U|^2dx\Big)^\frac 34\Big(\int\limits_{B(a)}|\nabla U|^2dx+\frac 1{a^2}\int\limits_{B(a)}|U|^2dx\Big)^\frac 34.$$
Taking into account estimates (\ref{1sts-s}) and (\ref{2nds-s}), we find
$$\int\limits_{B(a)}|U|^3dx\leq ca^{\frac 34m_1}(a^m+a^{m_1})^\frac 34<ca^{\frac 34(m_1+m)}$$
for any $a>1$. So,
\begin{equation}
	\label{l-3s-s}
	\sup\limits_{a>1}\frac 1{a^{\frac 34(m_1+m)}}\int\limits_{B(a)}|U|^3dx<\infty.	\end{equation}

\begin{pro}
	\label{simpleLiouville} Suppose that that the following additional conditions hold:
	\begin{equation}
		\label{1staddition}
	\frac 12< m	<\frac 35\quad(\Leftrightarrow 0<m_1<\frac 15)
	\end{equation}
	and
	\begin{equation}
		\label{2ndaddition}
		\sup\limits_{b\geq 1}\frac 1{b^{\gamma m_1}}\int\limits_{B(b)}|U|^2dy<\infty
	\end{equation}
	with $$0\leq \gamma < \frac {\ln {(2+m_1)}-\ln2}{m_1\ln 2}.$$
	Then $U(y)\equiv0$.
\end{pro}
\begin{proof} First of all, it is not difficult to deduce from (\ref{special-form}) the following energy identity (we have enough regularity):
$$0=\int\limits_{\mathbb R^3}\Big[|U|^2\Big(\frac 1{1+\alpha}y\cdot\nabla \psi+\frac {m_1}{1+\alpha}\psi\Big)+U\cdot\nabla\psi\Big(|U|^2+2P\Big)\Big]dy.$$
It is valid for any non-negative cut-off function $\psi\in C^\infty_0(\mathbb R^3)$.
We can pick up a cut-off function $0\leq \psi\leq 1$ so that $\psi=1$ in $B(R)$, $\psi=0$ out of $B(2R)$, and $|\nabla \psi|\leq c/R$ and $|y\cdot\nabla \psi|\leq 2$. Then, we have
$$\frac {m_1}{1+\alpha }\int\limits_{B(R)}|U|^2dy\leq  I_1+I_2+I_3,$$
$$I_1=-\frac 1{1+\alpha}\int\limits_{B(2R)\setminus B(R)}|U|^2y\cdot\nabla \psi dy,	\quad I_2=-\int\limits_{B(2R)}|U|^2U\cdot\nabla \psi dy,	$$
$$I_3=-2\int\limits_{B(2R)}PU\cdot\nabla \psi dy.$$
 Now, let us evaluate the first integral
 $$I_1\leq \frac 2{1+\alpha}(F(2R)-F(R)),$$
where
$$	F(R)=\int\limits_{B(R)}|U|^2dy.$$
To bound the second and the third integrals, we are going to use 	estimates \eqref{3rds-s} and \eqref{l-3s-s}:
$$	|I_2|+|I_3|\leq C(R^{\frac 34(3m-1)-1}+R^{\frac 14(3m-1)+\frac {4m}3-1})\leq \frac C{R^\beta},$$
where $\beta =\frac 54(1-\frac 53m)>0$.

Finally, we find
\begin{equation}
	\label{mainfor}
	F(R)\leq \theta F(2R)+\frac C{R^\beta}
\end{equation}
for all $R\geq 1$ with $\theta=\frac 2{2+m_1}.$
	
Iterations of \eqref{mainfor} gives
$$	F(R)\leq \theta^kF(2^kR)+C\frac {\theta^{k-1}}{R^\beta}\sum\limits^{k-1}_{i=0}\frac 1{(2^\beta\theta)^i}=
$$	$$=(\theta 2^{\gamma m_1})^kC+A(k,\theta,R),
$$
where 
$$A(k,\theta,R)=C\frac {\theta^{k-1}}{R^\beta}\sum\limits^{k-1}_{i=0}\frac 1{(2^\beta\theta)^i}.
$$	
By the assumption on $\gamma$, $\theta2^{\gamma m_1}<1$. So, $(\theta 2^{\gamma m_1})^kC\to 0$ as $k\to\infty$.
If we show that $A(k,\theta,R)\to 0$ as $k\to\infty$, then $F(R)=0$. Since it is true for all $R\geq 1$, we can conclude that $U\equiv0$.
Now, consider three cases. In the first one,
$2^\beta\theta>1$. Then,
$$A(k,\theta,R)=\frac {\theta^{k-1}}{R^\beta}\frac{1-\frac 1{(2^\beta\theta)^k}}{1-\frac 1{2^\beta\theta}}< \frac {\theta^{k-1}}{R^\beta}\frac{1}{1-\frac 1{2^\beta\theta}}\to0
$$ as $k\to\infty$.

In the second case, $2^\beta\theta<1.$ So,
$$A(k,\theta,R)=\frac {\theta^{k-1}}{R^\beta}\frac{1-\frac 1{(2^\beta\theta)^k}}{1-\frac 1{2^\beta\theta}}=\frac 1{R^\beta}\frac 1{2^{\beta(k-1)}}\frac{1-(2^\beta\theta)^k}{1-2^\beta\theta}\to 0
$$ as $k\to\infty$.

Finally, if $2^\beta\theta=1$,
$$A(k,\theta,R)=\frac {\theta^{k-1}}{R^\beta}k\to0$$ as $k\to\infty$. This completes our proof of the proposition.
\end{proof}

Similar arguments work for discrete self-similar solutions. We need to assume that our discrete self-similar solution is smooth enough for having the local energy identity instead of the local energy inequality. Then, we find
$$\int\limits^{S_0}_0\int\limits_{\mathbb R^3}\Big[|U(y,\tau)|^2\Big(\frac 1{1+\alpha}y\cdot\psi(y)+\frac {m_1}{1+\alpha}\psi(y)\Big)+U(y,\tau)\cdot\nabla \psi(y)\Big(|U(y,\tau)|^3+$$
$$+2P(y,\tau)\Big)\Big]dyd\tau=0
$$
for any non-negative $\psi\in C^\infty_0(\mathbb R^3)$. With the same choice of the cut-off  function $\psi$ as in the proof of Proposition \ref{simpleLiouville}, we shall have
$$F(R):=\int\limits_{B(R)}\int\limits_0^{S_0}|U(y,\tau)|^2dyd\tau\leq \theta F(2R)+\frac cR\int\limits^{S_0}_0\int\limits_{B(2R)}\Big(|U|^3+|U||P|)dyd\tau.$$

So, we need to evaluate integrals of $|U|^3$ and $|P|^\frac 32$. To this end we are going to exploit estimates \eqref{1st}, \eqref{2nd}, and \eqref{3rd} in the following way.
Given $0<\tau<S_0$ and $R>e^{\frac {2S_0}{1+\alpha}}$, we can find $a(\tau,R)>0$ such $b(a(\tau,R),\tau)=2R$. It is easy to check that $-\ln a^2<\tau<\infty$ and moreover $$2Re^{\frac {S_0}{1+\alpha}}>b(2R,\tau)>2R=b(a(\tau,R),\tau)\geq a(\tau,R)=2Re^{-\frac \tau{1+\alpha}}=
$$$$=2Re^{\frac \tau{1+\alpha}}e^{-2\frac \tau{1+\alpha}}\geq b(2R,\tau)e^{-2\frac {S_0}{1+\alpha}}\geq 2Re^{-2\frac {S_0}{1+\alpha}}>2.$$ 
In addition, for $0<\tau<S_0$, one has $-\ln(2R)^2<0<\tau$. Then,
 for example, it follows from \eqref{1st} that
$$\int\limits_{B(2R)}|U(y,\tau)|^2dy=\int\limits_{B(b(a(\tau,R),\tau)}|U(y,\tau)|^2dy\leq \int\limits_{B(b(2R,\tau))}|U(y,\tau)|^2dy\leq 
$$$$\leq Cb^{m_1}(2R,\tau)= CR^{m_1}$$ 
for all $0<\tau<S_0$ with a constant independent of $R$ and $\tau$.  

In the same way, one can treat the integral of $|\nabla U|^2$. Indeed,
$$
\int\limits_0^{S_0}\int\limits_{B(2R)}|\nabla U(y,\tau)|^2dyd\tau=$$$$=\int\limits_0^{S_0}b^{m}(a(\tau,R),\tau)\frac 1{b^{m}(a(\tau,R),\tau)}\int\limits_{B(b(a(\tau,R),\tau))}|\nabla U(y,\tau)|^2dyd\tau=$$$$=(2R)^{m}\int\limits_0^{S_0}\frac 1{b^{m}(a(\tau,R),\tau)}\int\limits_{B(b(a(\tau,R),\tau))}|\nabla U(y,\tau)|^2dyd\tau. $$
$$\leq c(S_0,\alpha,m)R^{m}\int\limits_0^{S_0}\frac 1{b^{m}(2R,\tau)}\int\limits_{B(b(2R,\tau))}|\nabla U(y,\tau)|^2dyd\tau\leq$$
$$\leq c(S_0,\alpha,m)R^{m}\int\limits_0^{S_0}\frac 1{b^{m}(2R,\tau)}\int\limits_{B(b(2R,\tau))}|\nabla U(y,\tau)|^2dyd\tau\leq$$
$$\leq c(S_0,\alpha,m)R^{m}\int\limits_{-\ln(2R)^2}^{\infty}\frac 1{b^{m}(2R,\tau)}\int\limits_{B(b(2R,\tau))}|\nabla U(y,\tau)|^2dyd\tau\leq CR^{m}.$$

Similar bound is valid for the pressure $P$:
$$\int\limits^{S_0}_0\int\limits_{B(2R)}|P|^\frac 32dyd\tau\leq CR^{2m}.
$$
 To find a bound of the integral of $|U|^3$, one  can use the above multiplicative inequality:
$$\int\limits^{S_0}_0\int\limits_{B(2R)}|U|^3dyd\tau\leq c\int\limits^{S_0}_0d\tau \Big(\int\limits_{B(2R)}|U|^2dy\Big)^\frac 34\Big(\int\limits_{B(2R)}|\nabla U|^2dy+$$$$+\frac 1{R^2}\int\limits_{B(2R)}|U|^2dy)\Big)^\frac 34\leq CR^{\frac 34m_1}\Big(R^m+R^{m_1}\frac{S_0}{R^2}\Big)^\frac 34\leq CR^{\frac 34(m+m_1)}.$$
 
Now, from the latter estimates, it follows  
  \eqref{mainfor}. The rest of the proof of $U\equiv0$ is the same as in Proposition \ref{simpleLiouville} if we make  similar additional  assumption:
	$$	\sup\limits_{b\geq 1}\frac 1{b^{\gamma m_1}}\int\limits_0^{S_0}\int\limits_{B(b)}|U|^2dy<\infty.$$
Here, a positive number  $\gamma$ obeys restrictins of Proposition \ref{simpleLiouville}.

\setcounter{equation}{0}
\section{Appendix I}
In this section, we are going to understand what kind of  Type II blowups is potentially possible provided the quantity $M^{s,l}_{\kappa,m_0}(v,R)$ is bounded. Consider potential blowups of Type II with the following property:
$$|v(x,t)|<c \Big(\frac 1{|x|^\alpha(-t)^\frac {(1-\alpha)2}2}\Big)^\gamma,
$$ where $\gamma>1$ and $0\leq \alpha\leq 1$. 

Now, our goal is to show that, given $0<m_0<1$,  $0< \alpha< 1$, one can find  numbers $ s$ and $l$ satisfying assumptions (\ref{add-cond-sl}) and a number $\gamma >1$
 such that 
 \begin{equation}
 	\label{bound6}\sup\limits_{0<R<1}M^{s,l}_{\kappa,m_0}(v,R)=\sup\limits_{0<R<1}R^{(1-m_0)\kappa}M^{s,l}_{\kappa}(v,R)<\infty.
 \end{equation}
To this end, let us
 pick up a number $\delta>0$ so that
\begin{equation}
\label{delta}	
1<\max\Big\{\frac 6{3+\alpha},\frac 4{3-\alpha}\Big\}<\delta(2-m_0)<2	
\end{equation} 
and fix it. Now, we let
$$\frac 1l=\delta\frac {2-m_0}2(1-\alpha),\quad\frac 1s=\delta\frac {2-m_0}3\alpha.
 $$ Now, by (\ref{delta}),
 $$\frac 3s+\frac 2l-\frac 32=\delta(2-m_0)-\frac 32<\frac 12.$$
 Next, again by (\ref{delta}),
  $$\frac 3s+\frac 2l-\frac 32=\delta(2-m_0)-\frac 32>\frac 12-\delta\frac {2-m_0}2(1-\alpha)=\frac 12-\frac 1l$$
 and 
 $$\frac 3s+\frac 2l-\frac 32=\delta(2-m_0)-\frac 32 >\delta \frac {2-m_0}4(1-\alpha)=\frac 1{2l}.$$
 So, numbers $s$ and $l$ obey conditions (\ref{add-cond-sl}).
  
  Now, we can introduce number $\gamma$, which by (\ref{add-cond-sl}) satisfying the following restriction
  \begin{equation}
  	\label{gamma6}1<\gamma=\Big(\frac 3s+\frac 2l-1\Big)(1-m_0)+1
  \end{equation}
  and as it follows from (\ref{delta})
  $$\delta(2-m_0)-\gamma=
  \delta(2-m_0)-((\delta(2-m_0) -1) (1-m_0)+1)=$$
  $$=\delta(2-m_0)m_0-m_0=m_0(\delta(2-m_0)-1)>0.$$
 But the latter implies the following inequalities 
 \begin{equation}
 	\label{integration}
 \frac 1l>\frac 12\gamma(1-\alpha), \quad\frac 1s>\frac 13\gamma\alpha.	\end{equation}

Now,  according to (\ref{integration}) and the choice of $\gamma$, the following calculations are legal:
$$M^{s,l}_{\kappa,m_0}(v,R)\leq cR^{-\kappa m_0}\int\limits_{-R^2}^0\Big(\int\limits_{B(R)}\Big(\frac 1
{|x|^\alpha(-t)^{\frac {(1-\alpha)}2}}\Big)^{s\gamma}dx\Big)^\frac lsdt$$
$$=cR^{-\kappa m_0}\int\limits^{R^2}_0t^{-\frac {1-\alpha}2\gamma l}dt\Big(\int^R_0r^{2-\alpha s\gamma}dr\Big)^\frac ls=
$$
$$=cR^{-\kappa m_0+(2-(1-\alpha)\gamma l)+(3-\alpha s\gamma)\frac ls}=c<\infty.
$$ 

Now, let us consider the case $\alpha=0$ and $\alpha=1$ in a similar way. The number $\gamma$ is going to be determined by the same formula (\ref{gamma6}).  Our aim is to show that, under a certain additional assumption on the number $m_0$, one can find numbers $s$ and $l$ satisfying (\ref{add-cond-sl}) and  (\ref{bound6}). 

For $\alpha=0$, assume that $\frac {19}{20}<m_0<1$, then one can let
$$s=\frac{10}{3
(1+\delta_1)},\quad l=\frac {20}{11(1-\delta_2)},
$$ where positive numbers $\delta_1$ and $\delta_2$ can be taken sufficiently small so that conditions (\ref{add-cond-sl}) hold and $l\gamma<2$.

If $\alpha=1$, the we assume that $\frac 45<m_0<1$ and let
$$s=\frac {15}{7(1-\delta_1)},\quad l=\frac {10}{3(1+\delta_2)}
$$
with sufficiently small positive numbers $\delta_1$ and $\delta_2$ chosen so that conditions (\ref{add-cond-sl}) hold and $s\gamma<3$.

\setcounter{equation}{0}
\section{Appendix II}

What we have for suitable weak solutions? There are three standard estimates, see \cite{Seregin2015} :
\begin{equation}
	\label{multiple2II}
	C_{1}(v,r)\leq c
	A^\frac 34_1(v,r)\Big(E_1(v,r)+A_1(v,r)\Big)^\frac 34
\end{equation}
for all $0<r\leq 1$;
\begin{equation}
	\label{pressure2II}
	D_1(q,r)\leq c\Big(\frac r\varrho D_1(q,\varrho)+\Big(\frac \varrho r\Big)^2C_1(v,\varrho)\Big)
\end{equation}
for all $0<r<\varrho\leq 1$;
\begin{equation}
	\label{energylocineqII}
	A_1(v,r)+E_1(v,r)\leq c\Big(C^\frac 23_1(v,2r)+C_1(v,2r)+D_1(q,2r)\Big)
\end{equation}
for all $0<r\leq 1/2$.

Let us modify (\ref{multiple2II}), using arguments from  papers \cite{SZ2006} and \cite{SS2018}. By H\"older inequality, we have
$$\int\limits_{B(\varrho)}|v|^3dx\leq \Big(\int\limits_{B(\varrho)}|v|^sdx\Big)^\lambda\Big(\int\limits_{B(\varrho)}|v|^2dx\Big)^\mu\Big(\int\limits_{B(\varrho)}|v|^6dx\Big)^\gamma, 
$$
where 
$$\lambda={\frac lq}\frac 1 s,\quad\gamma =\frac lq\Big(\frac{2}{s}+\frac 1l-1\Big),\quad
\mu=\frac lq\Big(\frac 3s+\frac 3l-2\Big),\quad q=2l\Big (\frac 3s+\frac 2l-\frac 32\Big)
$$ and numbers $s$ and $l$ satisfy restrictions (\ref{add-cond-sl}).

By Gagliardo-Nirenberg inequality, we then have
$$\int\limits_{B(\varrho)}|v|^3dx\leq c\varrho^{\mu n}A^\mu_{n}(v,\varrho)\Big(\int\limits_{B(\varrho)}|v|^sdx)^\lambda\Big(\int\limits_{B(\varrho)}(|\nabla v|^2+\frac 1{\varrho^2}|v|^2)dx\Big)^{3\gamma}
$$
and thus (since $3\gamma q'=1$ with $q'=\frac q{q-1}$)
$$\widetilde C_n(v,\varrho)=\frac 1{\varrho^{n+1}}\int\limits_{Q(\varrho)}|v|^3dz\leq c\varrho^{\mu n-n-1}A^\mu_{n}(v,\varrho) \Big(\int\limits^0_{-\varrho^2}\Big(\int\limits_{B(\varrho)}|v|^sdx\Big)^\frac 
lsdt\Big)^\frac 1q\times
$$
$$\times \varrho^{\frac n
{q'}}\Big(\frac 1{\varrho^n}\int\limits_{Q(\varrho)}|\nabla u|^2dz+\frac 1{\varrho^{n+2}}\int\limits_{Q(\varrho)}|u|^2dz\Big)^\frac 1{q'}=$$$$=c\varrho^{\mu n-n-1}\varrho^{\frac n
{q'}}\varrho^\frac {\kappa_n} qA^\mu_{n}(v,\varrho)(E_n(v,\varrho)+H_n(v,\varrho))^\frac 1{q'}(M_{\kappa_n}^{s,l}(v,\varrho))^\frac 1q=$$$$=cA^\mu_{n}(v,\varrho)(E_n(v,\varrho)+H_n(v,\varrho))^\frac 1{q'}(M_{\kappa_n}^{s,l}(v,\varrho))^\frac 1q.$$ 
Here,
$$H_n(v,\varrho)=\frac 1{\varrho^{n+2}}\int\limits_{Q(\varrho)}|u|^2dz, \quad M_{\kappa_n}^{s,l}(v,\varrho)= \frac 1{\varrho^{\kappa_n}}\int\limits^0_{-\varrho^2}\Big(\int\limits_{B(\varrho)}|v|^sdx\Big)^\frac lsdt,
$$
and
$$\kappa_n=n\kappa+q(1-n),\quad\kappa=
l\Big(\frac 3s+\frac 2l-1\Big).
$$

Let 
\begin{equation}
	\label{Mkappan}
	M_n=\sup\limits_{0<\varrho<1}M^{s,l}_{k_n}(v,\varrho)<\infty.
\end{equation}
 Then, 
$$\widetilde C_{n}(v,\varrho)\leq \varepsilon(E_{n}(v,\varrho)+A_{n}(v,\varrho))+c(q,\varepsilon)M_nA_{n}(v,\varrho)A^{\mu q}(v,\varrho).$$   Since $\mu q<1$, we find
$$\widetilde C_{n}(v,\varrho)\leq 2\varepsilon(E_{n}(v,\varrho)+A_{n}(v,\varrho))+c(s,l,\varepsilon)M_n^{\frac 1{1-\mu q}}.$$

In what follows, we are going to drop $v$ in our notation for energy quantities temporarily. Then local energy inequality (\ref{energylocineqII}) gives us:
$$A_{n}(R/2)+E_{n}(R/2)=(R/2)^{1-n}(A_1(R/2)+E_1(R/2))\leq$$$$\leq cR^{1-n}(C_1^\frac 23(R)+C_1(R)+D_1(R))=$$$$=c[R^{\frac 13(1-n)}\widetilde C_n^\frac 23(R)+\widetilde C_n(R)+\widetilde D_n(R)]\leq $$$$\leq c[\widetilde C_n^\frac 23(R)+\widetilde C_n(R)+\widetilde D_n(R)],$$
where $\widetilde D_n(R)=R^{n-1}D_n(R)$.

 It remains to scale the pressure decay estimate (\ref{pressure2II}), which leads to the following:
$$\widetilde D_n(\varrho)\leq c\Big[\Big(\frac \varrho r\Big)^{2-n}\widetilde D_n(r)+\Big (\frac r\varrho\Big)^{n+1}\widetilde C_n(r)\Big ]$$
for all $0<\varrho<r<1$.

Letting 
$$\mathcal E(\varrho)=A_n(\varrho)+E_n(\varrho)+\widetilde D_n(\varrho),
$$
we have 
$$A_{n}(\theta\varrho)+E_{n}(\theta\varrho)
\leq c[\widetilde C_n^\frac 23(2\theta\varrho)+\widetilde C_n(2\theta\varrho)+\tilde D_n(2\theta\varrho)]$$
and 
$$\widetilde D_n(\theta\varrho)\leq c\Big[\Big(\theta)^{2-n}\widetilde D_n(\varrho)+\Big(\frac 1\theta\Big)^{n+1}\widetilde C_n(\varrho)\Big ].$$
Elementary calculations show:
$$\mathcal E(\theta \varrho)\leq c\Big[\theta^{2-n}\widetilde D_n(\varrho)+c_1(\theta)\widetilde C_n^\frac 23(\varrho)+c_2(\theta)C_n(\varrho)\Big]\leq $$$$
c\Big[\theta^{2-n}\widetilde D_n(\varrho)+c_1(\theta)(2\varepsilon(E_{n}(v,\varrho)+A_{n}(v,\varrho))+c(s,l,\varepsilon)M_n^{\frac 1{1-\mu q}})^\frac 23+$$$$+ c_2(\theta)(2\varepsilon(E_{n}(v,\varrho)+A_{n}(v,\varrho))+c(s,l,\varepsilon)M_n^{\frac 1{1-\mu q}})\Big]\leq 
$$
$$\leq c\Big[(\theta^{2-n}+c_2(\theta)\varepsilon)\mathcal E(\varrho)+c_1(\theta)\varepsilon^\frac 23\mathcal E^\frac 23(\varrho)+c_3(s,l,\varepsilon,\theta,M_n)\Big]\leq $$
$$\leq c\Big[(\theta^{2-n}+c_2(\theta)\varepsilon +c_1^\frac 32(\theta)\varepsilon^\frac 12)\mathcal E(\varrho)+\varepsilon+
c_3(s,l,\varepsilon,\theta,M_n)\Big].$$

We can find $\theta$ so that
$$c\theta^{2-n}<\frac 14$$
and then $\varepsilon$ such that
$$c(c_2(\theta)\varepsilon +c_1^\frac 32(\theta)\varepsilon^\frac 12)<\frac 14$$
and thus 
$$\mathcal E(\theta\varrho)<\frac 12\mathcal E(\varrho)+c(\varepsilon+
c_3(s,l,\varepsilon,\theta,M_n)).$$ 
From the latter inequality, it follows that:
$$ 
\mathcal E(\theta^k\varrho )<\frac 1{2^k}\mathcal E(\varrho)+
c_4(s,l,\varepsilon,\theta,M_n)$$
for any $k=1,2,...$ and therefore
$$\mathcal E(r)<c(\theta)r^{c_6(\theta)}\mathcal E(1)+c_5(s,l,\varepsilon,\theta,M_n)$$
for all $0<r\leq 1$ with a positive constant $c_6(\theta)$. So, we have proved the following statement
\begin{pro}
	\label{additinalestimates}
	Let $v$ and $q$ be a suitable weak solution to the Navier-Stokes equations in $Q$. Assume that conditions (\ref{Mkappan}) and (\ref{add-cond-sl}). Then
$$A_n(\varrho)+E_n(\varrho)+\widetilde D_n(\varrho)\leq c$$	
for all $0<\varrho\leq 1$.	
	\end{pro}

\setcounter{equation}{0}
\section{Appendix III}



Here, we outline our proof of Theorem \ref{extension}.

 The statement has been already proved in \cite{SS2018}
  under the additional assumption 
  $$\frac 3s +\frac 2l-\frac 32>\max\Big\{\frac 1{2l},\frac 12-\frac 1l\Big\}.$$
So, we assume that
\begin{equation}
	\label{additional}
	\frac 3s +\frac 2l-\frac 32\leq\max\Big\{\frac 1{2l},\frac 12-\frac 1l\Big\}.
\end{equation}	
Let us consider three cases: $l=3$, $1\leq l<3$, and $l>3$. 

In the first case, it is easy to check that, from \eqref{additional}, it follows that $s\geq 3$. But then we have
$$C(R)\leq cM^{s,3}(R)\leq \mathcal M.$$
Then we repeat arguments from \cite{SS2018}.

In the second case where $1\leq l<3$, we deduce from \eqref{additional} that $s>3$. Then we let
$$\lambda =\frac 1{s-2},\quad\mu =\frac {s-3}{s-2} $$
and use the H\"older inequality in the following way
$$\int\limits_{B(R)}|v|^3dx\leq \Big(\int\limits_{B(R)}|v|^sdx\Big)^\lambda\Big(\int\limits_{B(R)}|v|^2dx\Big)^\mu.$$
Integration in $t$ and scaling give:
$$C(R)\leq c\mathcal M^\frac 1q A^\mu(R),$$
where $q=\frac l{\lambda s}$ and $0<\mu=\frac {s-3}{s-2}<1$. Then again we repeat arguments of \cite{SS2018}.

Now, consider the last case $l>3$. The condition \eqref{additional} says that $s\geq \frac 32$. Assume that that $s=\frac 32$. Then from \eqref{additional} it follows that $l=\infty$ and it contradicts with main assumption \eqref{weak13}. So, we must conlcude
that 
\begin{equation}
	\label{s} s>\frac 32.
\end{equation}
Now, let us consider two sub-cases: $s\geq 3$ and $\frac 32<s<3$.

In the first sub-case, we repeat arguments of the case $l=3$. The yield the following estimate
$$C(R)\leq c\Big(M^{s,l}(R)\Big)^\frac 3l\leq c\mathcal M^\frac 3l.$$

In the second sub-space, we let
$$0<\lambda =\frac 3{6-s}<1,\quad 0<\gamma=\frac{3-s}{6-s}<1.$$
By \eqref{s}, we have 
\begin{equation}
	\label{gamma}
	3\gamma<1.
\end{equation}
Introducing 
$q=\frac l{\lambda s}>1$ and observing that $3\gamma q'\leq 1$, where $q'=\frac q{q-1}$, we find after applications of H\"older  and Gagliardo-Nirenberg inequalities the following
$$\int\limits_{Q(R)}|v|^3dxdt\leq \int\limits^0_{-R^2}\Big(\int\limits_{B(R)}|v|^sdx\Big)^\lambda
\Big(\int\limits_{B(R)}|v|^6dx\Big)^\mu dt \leq 
$$
$$\leq \Big(\int\limits^0_{-R^2}\Big(\int\limits_{B(R)}|v|^sdx\Big)^{\lambda q}dt\Big)^\frac 1q
 \Big(\int\limits^0_{-R^2}\Big(\int\limits_{B(R)}|v|^6dx\Big)^{\gamma q'}dt\Big)^\frac 1{q'}\leq $$
$$\leq cR^\frac \kappa q\mathcal M^\frac 1q\Big(\int\limits^0_{R^2}\Big(\int\limits_{B(R)}(|\nabla v|^2+\frac 1{R^2}|v|^2)dx\Big)^{3\gamma q'}dt\Big)^\frac 1{q'}.$$
So, finally, we have 
$$C(R)\leq c\mathcal M^\frac 1q(E(R)+A(R))^{3\gamma}$$
with positive $\gamma$ satisfying \eqref{gamma}.
Then, in both sub-cases, it remains to repeat arguments of \cite{SS2018}.


\end{document}